\documentclass[epsfig,11pt]{llncs}
\usepackage{graphicx,latexsym,amsfonts}
\newcommand{\remove}[1]{}
\usepackage{amsfonts}

\newcommand{\bbV}{{\mathcal V}}
\newcommand{\bbI}{{\mathcal I}}

\newcommand{\bbH}{{\mathcal H}}
\usepackage{epsfig,latexsym,graphicx}
\newtheorem{observation}{Observation}
\begin{document}

\title{Holes or Empty Pseudo-Triangles in Planar Point Sets}

\author{Bhaswar B. Bhattacharya$^1$\and Sandip Das$^2$}

\institute{
$^1$ \small{Department of Statistics, Stanford University, California, USA,\\
{\tt bhaswar.bhattacharya@gmail.com}}\\
$^2$ \small{Advanced Computing and Microelectronics Unit, Indian Statistical Institute,\\
Kolkata, India, {\tt sandipdas@isical.ac.in}}}

\maketitle

\begin{abstract}
Let $E(k, \ell)$ denote the smallest integer such that any set of at least $E(k, \ell)$ points in the plane, no three on a line, contains either an empty convex polygon with $k$ vertices or an empty pseudo-triangle with $\ell$ vertices.
The existence of $E(k, \ell)$ for positive integers $k, \ell\geq 3$, is the consequence of a result proved by Valtr [{\it Discrete and Computational Geometry}, Vol. 37, 565--576, 2007]. In this paper, following a series of new results about the existence of empty pseudo-triangles in point sets with triangular convex hulls, we determine the exact values of $E(k, 5)$ and $E(5, \ell)$, and prove bounds on $E(k, 6)$ and $E(6, \ell)$, for $k, \ell\geq 3$.
By dropping the emptiness condition, we define another related quantity $F(k, \ell)$, which is the smallest integer such that any set of at least $F(k, \ell)$ points in the plane, no three on a line, contains a convex polygon with $k$ vertices or a pseudo-triangle with $\ell$ vertices. Extending a result of Bisztriczky and T\'oth [{\it Discrete Geometry,
Marcel Dekker}, 49--58, 2003], we obtain the exact values of $F(k, 5)$ and $F(k, 6)$, and obtain
non-trivial bounds on $F(k, 7)$.\\

\noindent\textbf{Keywords.} Convex hull, Discrete geometry, Empty convex polygons,
Erd\H os-Szekeres theorem, Pseudo-triangles, Ramsey-type results.
\end{abstract}

\section{Introduction}
\label{sec:intro}

The famous Erd\H os-Szekeres theorem \cite{erdos} states that for every positive integer {\it m}, there exists a smallest integer $ES(m)$, such that any set of at least $ES(m)$ points in the plane, no three on a line, contains $m$ points which lie on
the vertices of a convex polygon. Evaluating the exact value of $ES(m)$
is a long standing open problem. A construction due to Erd\H os
\cite{erdosz} shows that $ES(m)\geq 2^{m-2} +1$, which is
conjectured to be sharp. It is known that $ES(4)=5$ and $ES(5)=9$
\cite{kalb}. Following a long computer search, Szekeres and Peters
\cite{szekeres} recently proved that $ES(6)=17$. The value of
$ES(m)$ is unknown for all $m> 6$. The best known upper bound for
$m\geq 7$ is due to T\'oth and Valtr \cite{tothvaltr}: $ES(m)
\leq {{2m-5}\choose {m-3}}+1$.

In 1978 Erd\H os \cite{erdosempty} asked whether for every positive integer $k$,
there exists a smallest integer $H(k)$, such that any set of at least
$H(k)$ points in the plane, no three on a line, contains $k$ points which lie on the vertices of a
convex polygon whose interior contains no points of the set. Such a subset is called an {\it empty
convex $k$-gon} or a {\it k-hole}. Esther Klein showed {$ H(4)=5$} and Harborth \cite{harborth}
proved that {$ H(5)=10$}. Horton \cite{horton} showed
that it is possible to construct arbitrarily large set of points
without a 7-hole, thereby proving that {$ H(k) $} does not exist
for {$ k \geq 7$}. Recently, after a long wait, the existence of $ H(6)$ has been proved
by Gerken \cite{gerken} and independently by Nicol\'as
\cite{nicolas}. Later, Valtr \cite{valtrhexagon} gave a simpler version of Gerken's proof.

These problems can be naturally generalized to polygons that are not necessarily
convex. In particular, we are interested in pseudo-triangles, which are
considered to be the natural counterpart of convex polygons. A pseudo-triangle
is a simple polygon with exactly three vertices having interior angles less
than $180^{\circ}$. A pseudo-triangle with $\ell$ vertices is called a
$\ell$-pseudo-triangle. A set of points is said to contain an empty $\ell$-pseudo-triangle
if there exists a subset of $\ell$ points forming a $\ell$-pseudo-triangle which contains
no point of the set in its interior. A pseudo-triangle with $a, b, c$ as the convex vertices has three concave side chains between the vertices $a, b$  and $b, c$, and $c, a$. Based on the length of the three side chains, a pseudo-triangle can be distinguished into three types: a {\it standard} pseudo-triangle,
if each side chain has at least two edges, a {\it mountain}, if
exactly one side chain has only one edge, and a {\it fan}, if exactly two side
chains consists of only one edge (Figure \ref{fig1}). The {\it apex} of a {\it fan}
pseudo-triangle is the convex vertex having exactly one edge in both
its incident side chains.

\begin{figure*}[h]
\centering
\begin{minipage}[c]{1.0\textwidth}
\centering
\includegraphics[width=5.25in]
    {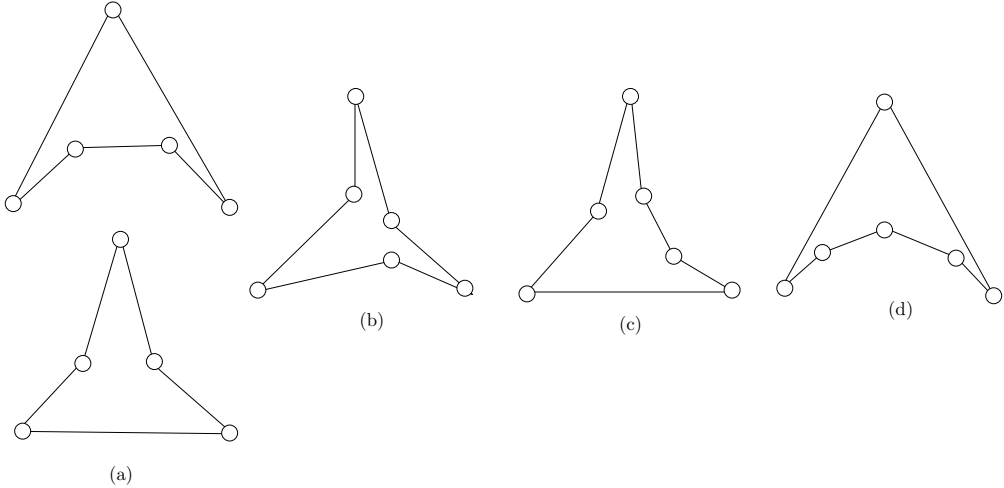}\\
\end{minipage}%
\caption{Pseudo-triangles: (a) Types of 5-pseudo-triangles, (b) Standard 6-pseudo-triangle, (c)
6-mountain, (d) 6-fan.}
\label{fig1}
\end{figure*}

In spite of considerable research on the various combinatorial and algorithmic
aspects of pseudo-triangles \cite{surveypseudo}, little is known about the existence of empty
pseudo-triangles in planar point sets. Kreveld and Speckmann \cite{empty_pt_speckmann}
devised techniques to analyze the maximum and minimum number of empty
pseudo-triangles defined by any planar point set. Ahn et al. \cite{optimal_empty_pseudo_traingle} considered the optimization problems of computing an empty pseudo-triangle with minimum perimeter, maximum
area, and minimum longest concave chain.

In this paper, analogous to the quantity $H(k)$, we define the Ramsey-type quantity $E(k, \ell)$ as
the smallest integer such that any set of at least $E(k, \ell)$ points in the plane, no
three on a line, contains a $k$-hole or an empty $\ell$-pseudo-triangle. The
existence of $E(k, \ell)$ for all $k, \ell\geq 3$, is a consequence of a result proved by Valtr \cite{valtropencupscaps}, and later by C\v ern\'y \cite{cerny}.

\begin{theorem}[\cite{cerny,valtropencupscaps}]
For any $k,~\ell\leq 3$, there is a least integer $n(k, \ell)$ such that any point
$p$ in any set $S$ of size at least $n(k, \ell)$, in general position, is the apex of an
empty $k$-fan in $S$ or it is one of the vertices of a $\ell$-hole in $S$.
\label{th:valtr}
\end{theorem}

Note that $E(k, \ell)\leq n(k, \ell)\leq 2^{{k+\ell-2}\choose{k+1}}+1$, where  the upper bound on $n(k, \ell)$ follows from Valtr \cite{valtropencupscaps}. However, the upper bound  $n(k, \ell)$ is double
exponential in $k+\ell$. In this paper, following the long and illustrious history of the quantities
$ES(k)$ and $H(k)$, we consider the problem of evaluating the exact values of $E(k, \ell)$ for small values of $k$ and $\ell$. Following a series of new results regarding the existence of empty pseudo-triangles in point sets with triangular convex hulls, we determine new bounds on $E(k, \ell)$ for small values of $k$ and $\ell$. We begin by proving that any set of points with three points on the convex hull and at least two, three, or five interior points
always contains an empty 5-pseudo-triangle, an empty 6-pseudo-triangle, or an empty 7-pseudo-triangle, respectively.
Using these three results and other results from the literature, we determine the exact  values of $E(k, 5)$
and $E(5, \ell)$, for all $k, \ell \geq 3$. We also obtain bounds on $E(k, 6)$ and $E(\ell, 6)$, for different values of $k$ and $\ell$ and discuss other implications of our results.

If the condition of emptiness is dropped from $E(k, \ell)$ we get another related quantity $F(k, \ell)$.
Let $F(k, \ell)$ be the smallest integer such that any set of at least
$F(k, \ell)$ points in the plane, no three on a line, contains a convex polygon with
$k$ vertices or a $\ell$-pseudo-triangle. From the Erd\H os-Szekeres theorem
it follows that $F(k, \ell)\leq ES(k)$ for all $k, \ell \geq 3$. Evaluating
non-trivial bounds of $F(k, \ell)$ is also an interesting problem. While addressing a different
problem Aichholzer et al. \cite{toth} showed that $F(6, 6)=12$.  In this paper, using our results on empty-pseudo-triangles and extending a result of Bisztriczky and Fejes T\'oth \cite{fejestoth}, we show that
$F(k, 5)=2k-3$, $F(k, 6)=3k-6$. We also obtain non-trivial bounds on $F(k, 7)$, for $k\geq 3$. Finally, we obtain the exact value of $F(5, \ell)$ and new bounds on $F(6, \ell)$, for $\ell\geq 3$.

The paper is organized as follows. In Section \ref{sec:nd} we introduce the required notations and definitions. In Section \ref{sec:preliminary_pt} we prove two preliminary observations. The results regarding the existence of empty pseudo-triangles in point sets with triangular convex hulls are presented in Section \ref{sec:triangle}. The bounds on $E(k, \ell)$ and $F(k, \ell)$ are presented in Section \ref{sec:e(k,l)} and Section \ref{sec:f(k,l)}, respectively. In Section \ref{conclusion} we summarize our results and give directions for future works.


\section{Notations and Definitions}
\label{sec:nd}

We first introduce the definitions and notations required for the remaining part of the paper.
Let $S$ be a finite set of points in the plane in general position, that is, no three on a line.
Denote the {\it convex hull} of $S$ by $CH(S)$. The
boundary vertices of $CH(S)$, and the points of $S$ in the interior
of $CH(S)$ are denoted by $\bbV(CH(S))$ and $\tilde{\bbI}(CH(S))$, respectively. A region $R$ in the plane is
said to be {\it empty} in $S$ if $R$ contains no elements of $S$ in its interior. Moreover, for any set $T$,
$|T|$ denotes the cardinality of $T$.

By $\mathcal P:=p_1p_2\ldots p_m$ we denote the region bounded by the simple polygon with vertices $\{p_1, p_2, \ldots, p_m\}$ ordered anti-clockwise. Let $\bbV(\mathcal P)$ denote the set of vertices $\{p_1, p_2, \ldots, p_m\}$ and $\bbI(\mathcal P)$ the interior of $\mathcal P$. A simple polygon $\mathcal P_0$ is {\it contained} in a simple polygon $\mathcal P$ if $\mathcal V(\mathcal P_0)\subseteq \mathcal V(\mathcal P)$ and $\mathcal I(\mathcal P_0)\subseteq \mathcal I(\mathcal P)$.

For any three points $p, q, r \in S$, $\bbH(pq, r)$ denotes the open half plane bounded by the line $pq$ containing the point $r$. Similarly, $\mathcal H_c(pq, r)$ denotes the closed half plane bounded by the line $pq$ containing the point $r$. Similarly, $\overline{\bbH}(pq, r)$ is the open half plane bounded by the line $pq$ not containing the point $r$.

The {\it $j$-th convex layer} of $S$, denoted by $L\{j, S\}$, is the
set of points of $S$ that lie on the boundary of $CH(S\backslash\{\bigcup_{i=1}^{j-1}L\{i, S\}\})$, where $L\{1, S\}=\bbV(CH(S))$. 

Moreover, if $\angle rpq < \pi$, $Cone(rpq)$ denotes the interior of the
angular domain $\angle rpq$. A point $s \in
Cone(rpq)\cap S$ is called the {\it nearest angular neighbor} of
$\overrightarrow{pq}$ in $Cone(rpq)$ if $Cone(spq)$ is empty in $S$.
Similarly, for any convex region $R$ a point $s \in R\cap S$ is called the {\it nearest angular neighbor} of
$\overrightarrow{pq}$ in $R$ if $Cone(spq)\cap R$ is empty in $S$. Also,
for any convex region $R$, the point $s\in S$, which has the shortest perpendicular distance to the line segment $pq$, $p, q\in S$, is called the {\it nearest neighbor} of $pq$ in $R$.

\section{Empty Pseudo-Triangles: Preliminary Observations}
\label{sec:preliminary_pt}



A pseudo-triangle with vertices $a, b, c$ of the convex hull has three concave side chains between
the pair of vertices $a, b$  and $b, c$, and $c, a$.
We denote the vertices of the pseudo-triangle lying on the concave side chain between $a$ and $b$
by $C(a, b)$. Similarly, we denote by $C(b,c)$ and $C(c,
a)$, the vertices on the concave side chains between $b, c$ and $c, a$, respectively.

In this section, we prove two observations about transformation and reduction of pseudo-triangles.

\begin{figure*}[h]
\centering
\begin{minipage}[c]{0.33\textwidth}
\centering
\includegraphics[width=2.7in]
    {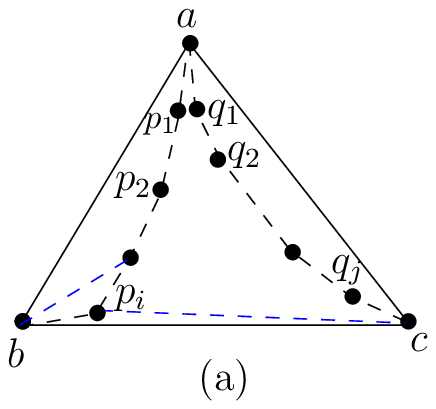}\\
\end{minipage}%
\begin{minipage}[c]{0.33\textwidth}
\centering
\includegraphics[width=2.7in]
    {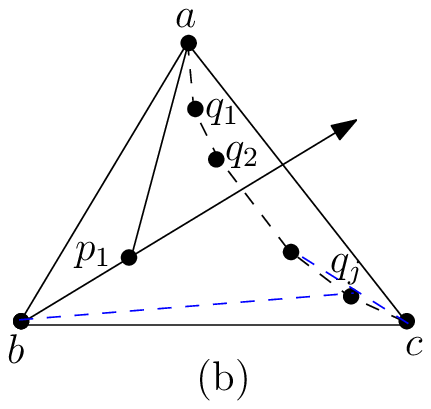}\\
\end{minipage}
\begin{minipage}[c]{0.33\textwidth}
\centering
\includegraphics[width=2.6in]
    {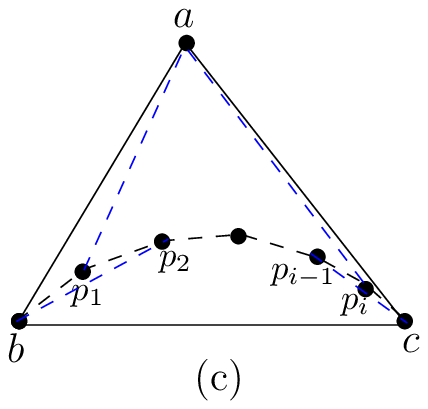}\\
\end{minipage}%
\caption{Illustration for the proofs of Observation \ref{ob:pt} and Observation \ref{ob:mm}.}
\label{fig2}
\end{figure*}

\begin{observation}
Any $\ell$-pseudo-triangle can transformed to a standard $\ell$-pseudo-triangle, for every $\ell\geq 6$, by
appropriate insertion and deletion of edges.
\label{ob:pt}
\end{observation}

\begin{proof}
Let $\mathcal P$ be a non-standard $\ell$-pseudo-triangle with $\ell\geq 6$, having convex
vertices $a, b, c$. Then, we have the following two cases:

\begin{description}
\item[{\it Case} 1:] Let $\mathcal P$ is a $\ell$-mountain with convex chains
$C(a, b)=\{a,$ $p_1, p_2, \ldots, p_i, b\}$, $C(b, c)=\{b,c\}$, and
$C(a, c)=\{a, q_1, q_2, \ldots, q_j, b\}$, such that $i+j+3=\ell$,
arranged as shown in Figure \ref{fig2}(a). Let $s_\alpha$ be the
nearest neighbor of $bc$ in $C(a, b)\cup C(a, c)$.  Then,
$\{b, s_\alpha, c\}$ are the vertices of a concave chain. If $i, j >1$, then both $|C(a, b)\backslash \{s_{\alpha}\}|\geq 1$ and $|C(a, c)\backslash \{s_{\alpha}\}|\geq 1$, and w.l.o.g. we can assume that
$s_{\alpha}\in C(a, b)$. In this case $s_{\alpha}=p_i$ and $\{a, p_1, p_2, \ldots, p_{i-1}, b\}$,
$\{b, p_i, c\}$, and $\{a, q_1, q_2, \ldots, q_j, c\}$ are the vertices of the convex chains
which form a standard $\ell$-pseudo-triangle as shown in  Figure \ref{fig2}(a).
So, w.l.o.g. it suffices to consider the case $i=1$ (Figure \ref{fig2}(b)). If $Cone(p_1bc)$
contains a point of $C(a, c)\backslash\{a, c\}$, then $\{a, p_1,
b\}$, $\{b, q_j, c\}$, and  $\{a, q_1, q_2, \ldots, q_{j-1}, b\}$
are the vertices of the three concave chains of a standard $\ell$-pseudo-triangle.
Otherwise, all the points of $C(a,c)\backslash\{a, c\}$ are in $Cone(abp_1)$,
and $\{a, q_1, b\}$, $\{b, p_1, c\}$, and $\{a, q_2,q_3,\ldots, q_i, c\}$ are the vertices of the three concave chains of a standard $\ell$-pseudo-triangle.

\item[{\it Case} 2:] Let $\mathcal P$ is a $\ell$-fan with $C(a, b)=\{a, b\}$, $C(b,c)=\{b, p_1, p_2, \ldots, p_i, c\}$ and  $C(a, c)=\{a, b\}$, where $i+3=\ell$, as shown in Figure \ref{fig2}(c). Then, the $\ell$-pseudo-triangle with concave chains formed by the
set of vertices $\{a, p_1, b\}$, $\{b, p_2, p_3, \ldots, p_{i-1}, c\}$,
and  $\{a, p_i, b\}$ is standard (Figure \ref{fig2}(c)). \hfill $\Box$
\end{description}
\end{proof}

\begin{observation}
An empty $\ell$-mountain contains an empty $m$-mountain whenever $3\leq m < \ell$.
\label{ob:mm}
\end{observation}
\begin{proof}
We need to show that every empty $\ell$-mountain contains an empty $(\ell-1)$-mountain for any $\ell\geq 4$.
Let $\mathcal P$ be a $\ell$-mountain with $\ell\geq 4$, having convex vertices $a, b, c$. Let
$C(a, b)=\{a,p_1, p_2, \ldots, p_i, b\}$, $C(b, c)=\{b,c\}$, and  $C(a, c)=\{a, q_1, q_2, \ldots, $ $q_j, b\}$ be
the vertices of the three concave chains of $\mathcal P$, such that $i+j+3=\ell$, as shown in Figure \ref{fig2}(a).
If both $i, j >1$, then we can obtain an empty $(\ell-1)$-mountain by taking the nearest neighbor of $bc$
in $C(a,b)\cup C(a,c)$ and removing either $b$ or $c$.

Otherwise, w. l. o. g. assume that $i=1$. If $Cone(p_1bc)\cap (C(a, c)\backslash\{a, c\})$ is non-empty, that is,
$q_j \in Cone(p_1bc)\cap (C(a, c)\backslash\{a, c\})$, then $\{a, p_1, b\}$, $\{b, q_j\}$, and $\{a, q_1, q_2, \ldots, q_{j}\}$ forms an empty $(\ell-1)$-mountain (Figure \ref{fig2}(b)). Similarly, if $q_j \in Cone(abp_1)\cap (C(a, c)\backslash\{a, c\})$, then $\{b, p_1, q_1\}$, $\{b, c\}$, and $\{q_1,q_2,\ldots q_j, c\}$ form an empty $(\ell-1)$-mountain.\hfill $\Box$
\end{proof}

\section{Empty Pseudo-Triangles in Point Sets with Triangular Convex Hulls}
\label{sec:triangle}

In this section we prove three results about the existence of empty pseudo-triangles in point sets with triangular convex hulls. These results will be used later to obtain bounds on $E(k, \ell)$ and $F(k, \ell)$.

\subsection{Empty 5-Pseudo-Triangle}
\label{sec:empty5pt}

\begin{lemma}
Any set $S$ of points in the plane in general position with $|CH(S)|=3$ and $|\tilde{\bbI}(CH(S))|\geq 2$
contains an empty 5-pseudo-triangle.
\label{lm:lm1}
\end{lemma}

\begin{figure*}[h]
\centering
\begin{minipage}[c]{1.0\textwidth}
\centering
\includegraphics[width=2.0in]
    {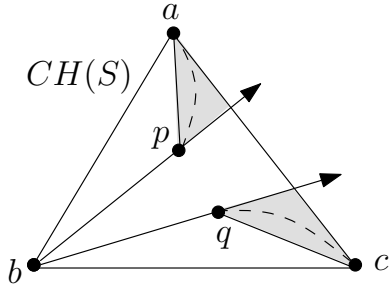}\\
\end{minipage}%
\caption{Illustration for the proof of Lemma \ref{lm:lm1}.}
\label{figure_e(5)}
\end{figure*}

\begin{proof}Let $\bbV(CH(S))=\{a, b, c\}$, where the vertices taken in counter-clockwise order.
Consider two points $p, q\in \tilde{\bbI}(CH(S))$, which are consecutive in
the radial order around the vertex $b$ of $\bbV(CH(S))$, that is, $Cone(pbq)$
is empty in $S$.
Let $C_p=\bbV(CH(\mathcal H_c (bp, a)\cap S))$ and
$C_q=\bbV(CH(\mathcal H_c (bq, c)\cap S))$ (Figure \ref{figure_e(5)}). Observe that
$C_p\cup C_q$ form an empty $\ell$-mountain with $\ell\geq 5$.
The existence of an empty 5-pseudo-triangle now follows from Observation \ref{ob:mm}. \hfill $\Box$
\end{proof}

\subsection{Empty 6-Pseudo-Triangle}
\label{sec:empty6pt}

\begin{lemma}
Any set $S$ of points in the plane in general position with $|CH(S)|=3$ and $|\tilde{\bbI}(CH(S))|\geq 3$
contains an empty standard 6-pseudo-triangle.
\label{lm2}
\end{lemma}

\begin{figure*}[h]
\centering
\begin{minipage}[c]{0.33\textwidth}
\centering
\includegraphics[width=1.8in]
    {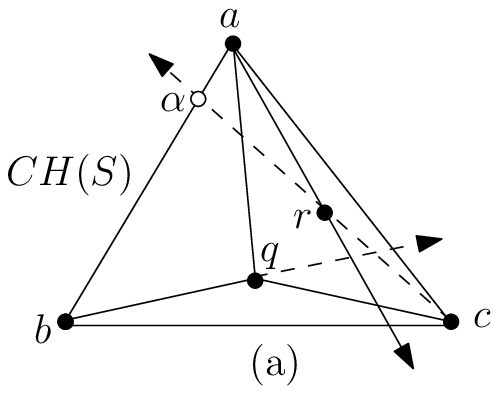}\\
\end{minipage}%
\begin{minipage}[c]{0.3\textwidth}
\centering
\includegraphics[width=1.8in]
    {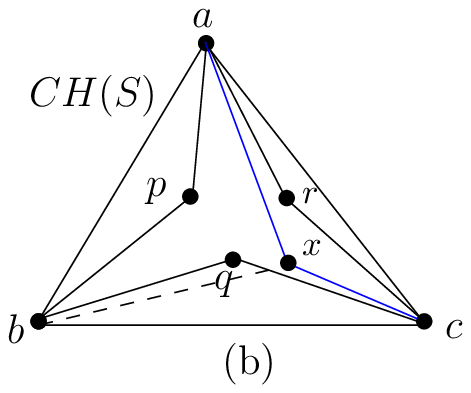}\\
\end{minipage}
\begin{minipage}[c]{0.33\textwidth}
\centering
\includegraphics[width=1.8in]
    {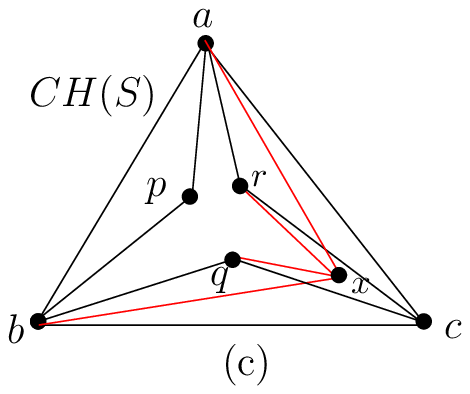}\\
\end{minipage}
\caption{Illustration for the proof of Lemma \ref{lm2}}
\label{fig3}\vspace{-0.2in}
\end{figure*}

\begin{proof}Let $\bbV(CH(S))=\{a, b, c\}$, with the vertices taken in counter-clockwise order.
Suppose that $|\tilde{\bbI}(CH(S))|=\{p, q, r\}$, and $q$ be such that $\bbI(qbc)$ is empty in $S$ (Figure \ref{fig3}(a)). When both $\bbI(qab)$ and $\bbI(qac)$ are non-empty in $S$, either $apbqcr$ or $arbqcp$ forms
an empty 6-pseudo-triangle. Therefore, w. l. o. g. assume that
$\bbI(qab)\cap S$ is empty and $p, r\in\bbI(qac)\cap S$. Let $r$ be the first angular neighbor
of $\overrightarrow{ac}$ in $Cone(qac)$ and $\alpha$ be
the point where $\overrightarrow{cr}$ intersects the boundary of $CH(S)$.
If $p\in Cone(ar\alpha )$,  then $aprcqb$ is an empty 6-pseudo-triangle.
Otherwise, $Cone(ar\alpha)$ is empty and
either $arcpbq$ or $arcqbp$ is an empty 6-pseudo-triangle. This empty
pseudo-triangle can be transformed to an empty standard 6-pseudo-triangle
by Observation \ref{ob:pt}.

Next, suppose that there are more than three points in $\tilde{\bbI}(CH(S))$.
It follows from the above that there are three points $p, q, r\in \tilde{\bbI}(CH(S))$ such that
$\mathcal A_1=apbqcr$ is a standard 6-pseudo-triangle with minimal number of interior points.

If $\mathcal A_1$ is not empty, there exists a point $x\in S$ in the interior of $\mathcal A_1$. The
three line segments $xa$, $xb$, and $xc$ may or may not intersect the boundary of $\mathcal A_1$.
If any two of these line segments, say $xa$ and $xc$, do not intersect the edges of $\mathcal A_1$, then
$\mathcal A_2=apbqcx$ is a standard 6-pseudo-triangle which is contained in
$\mathcal A_1$ (Figure \ref{fig3}(b)). Otherwise, there are two segments, say
$xa$ and $xb$, which intersect the edges of $\mathcal A_1$. In this case,
$\mathcal A_2=apbqxr$ is a standard 6-pseudo-triangle contained in $\mathcal A_1$, containing less interior points than $\mathcal A_1$ (Figure \ref{fig3}(c)). This contradicts the minimality of $\mathcal A_1$ and implies that $\mathcal A_1$ is empty in $S$. \hfill $\Box$
\end{proof}

\subsection{Empty 7-Pseudo-Triangles}
\label{sec:e(7)}

Let $S$ be a set of points in the plane in general position.
An interior point $p\in S$ is called a $(x,y, z)-splitter$ of $CH(S)$ if $|\bbV(CH(S))| = 3$ and the three triangles formed inside $CH(S)$ by the three line segments $pa$, $pb$, and $pc$ contain $x\geq y \geq z$ interior points of $S$.

We use this definition to establish a sufficient condition for the
existence of an empty 7-pseudo-triangle in sets having triangular
convex hull.

\begin{theorem}Any set $S$ of points in the plane in general position with $|CH(S)|=3$ and $|\tilde{\bbI}(CH(S))|\geq 5$ contains an empty 7-pseudo-triangle. Moreover, there
exists a set $S$ with $|CH(S)|=3$ and $|\tilde{\bbI}(CH(S))|=4$, that does not contain a 7-pseudo-triangle.
\label{lm:e7}
\end{theorem}

\begin{figure*}[h]
\centering
\begin{minipage}[c]{0.33\textwidth}
\centering
\includegraphics[width=2.6in]
    {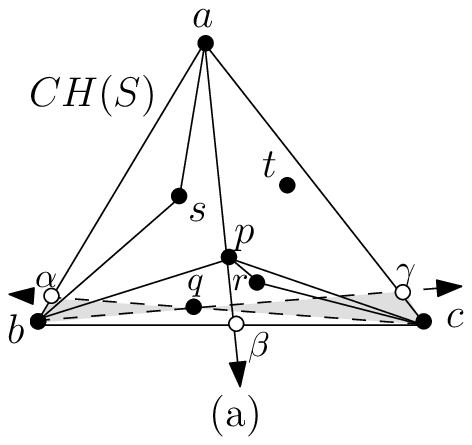}\\
\end{minipage}%
\begin{minipage}[c]{0.33\textwidth}
\centering
\includegraphics[width=2.6in]
    {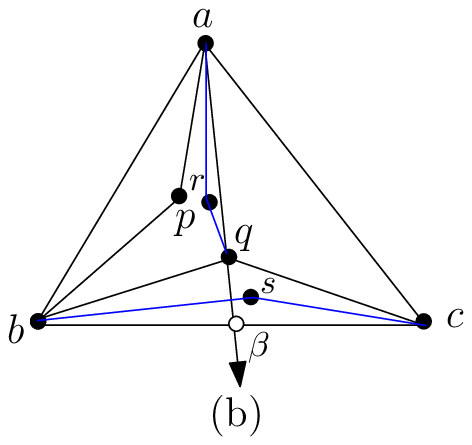}\\
\end{minipage}
\begin{minipage}[c]{0.33\textwidth}
\centering
\includegraphics[width=2.57in]
    {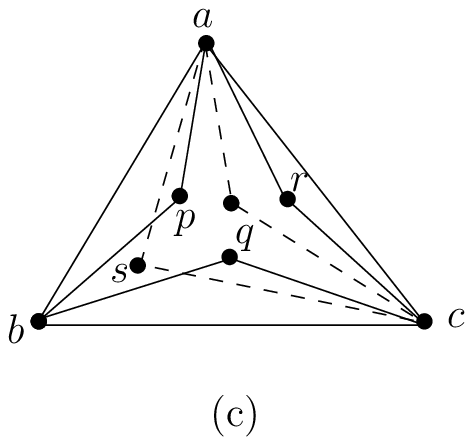}\\
\end{minipage}%

\caption{Illustration for the proof of Lemma \ref{lm:e7_lemma}.}
\label{fig:e(7)existence}
\end{figure*}

\subsubsection{Proof of Theorem \ref{lm:e7}}

We begin the proof of Theorem \ref{lm:e7} with the following lemma:

\begin{lemma}Any set $S$ of points in the plane, in general position, with $|CH(S)|=3$ and $|\tilde{\bbI}(CH(S))|\geq 5$,
contains a 7-pseudo-triangle.
\label{lm:e7_lemma}
\end{lemma}

\begin{proof}
Let $\bbV(CH(S))=\{a, b, c\}$ with the vertices taken in counter-clockwise order. Since we have to find a 7-pseudo-triangle, which is not necessarily empty, it suffices to assume that $|\tilde{\bbI}(CH(S))|=5$.

Assume that $p\in \tilde{\bbI}(CH(S))$ is such that $\bbI(pab)$, $\bbI(pbc)$, and $\bbI(pca)$ are all non-empty in $S$. Therefore, $p$ must be a $(2,1,1)$-splitter of $CH(S)$. Without loss of generality, let $q, r\in \bbI(pbc)\cap S$,
$s\in \bbI(pab)\cap S$, and $t\in \bbI(pac)\cap S$ be such that $q$ is the nearest
angular neighbor of $\overrightarrow{bc}$ in $\bbI(pbc)$. Let $\alpha, \beta, \gamma$ be the points where
$\overrightarrow{cq}, \overrightarrow{ap}, \overrightarrow{bq}$ intersect the boundary of $CH(S)$, respectively.
Let $R_1=\bbI(bq\alpha)\cap \bbI(bpc)$ and $R_2=\bbI(cq\gamma)\cap \bbI(bpc)$ (see Figure \ref{fig:e(7)existence}(a)).
If $r\in R_1\cup R_2$, then $asbqrcp$ or $asbrqcp$ is a 7-pseudo-triangle.
Thus, assume that $(R_1\cup R_2)\cap S$ is empty. If $r\in \bbI(\beta pc)\cap S$, then $asbqcrp$ is a 7-pseudo-triangle.
Otherwise, $r\in\bbI(\beta pb)\cap S$, and $aprbqct$ is a 7-pseudo-triangle.

Therefore, suppose that none of the interior points of $CH(S)$ is a $(2,1, 1)$-splitter of $CH(S)$.
From the proof of Lemma \ref{lm2} it is clear that the three vertices of $CH(S)$ along with any three points $p, q, r\in \tilde{\bbI}(CH(S))$ form a standard 6-pseudo-triangle $\mathcal P=apbqcr$. This 6-pseudo-triangle has the vertices of $CH(S)$ as the three convex vertices, and it is not necessarily empty. Now, there are two cases:

\begin{description}
\item[{\it Case} 1:] $\mathcal P$ is empty in $S$. The remaining two points $s$ and $t$ in $\tilde{\bbI}(CH(S))$, must be in some of the three triangles - $ pab$, $ qbc$, and $ rca$. W. l. o. g., assume that
$s\in \bbI(qbc)\cap S$. Since $q$ is not a $(2,1,1)$-splitter, either $\bbI(qab)\cap S$ or $\bbI(qac)\cap S$
is empty in $S$. If $\bbI(qac)\cap S$ is empty, $apbscqr$ is a 7-pseudo-triangle (Figure \ref{fig:e(7)existence}(b)). Otherwise, $\bbI(qab)$ is empty in $S$ then $apqbscr$ is a 7-pseudo-triangle. .

\item[{\it Case} 2:]$\mathcal P$ is non-empty in $S$. Let $s\in \bbI(\mathcal P)\cap S$. If any one of three line segments $sa$, $sb$, or
$sc$ intersects the boundary of $\mathcal P$ we get a 7-pseudo-triangle. Otherwise, two of these three segments go directly, and we have a smaller 6-pseudo-triangle with $a, b, c$ as its convex vertices (Figure \ref{fig:e(7)existence}(c)).
Continuing in this way, we finally get a 7-pseudo-triangle or an empty 6-pseudo-triangle with $a, b, c$ as its convex vertices, which then reduces to {\it Case} 1. \hfill $\Box$
\end{description}
\end{proof}

\begin{figure*}[h]
\centering
\begin{minipage}[r]{0.5\textwidth}
\centering
\includegraphics[width=3.5in]
    {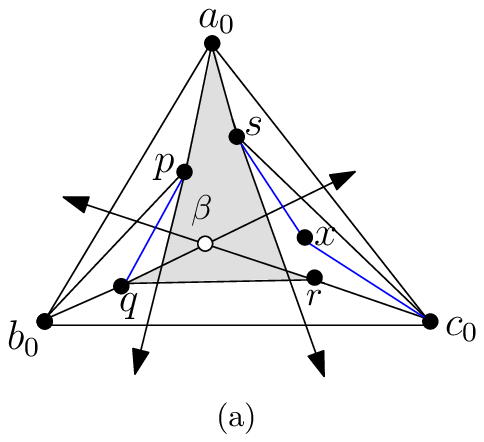}\\
\end{minipage}%
\begin{minipage}[c]{0.5\textwidth}
\centering
\includegraphics[width=3.5in]
    {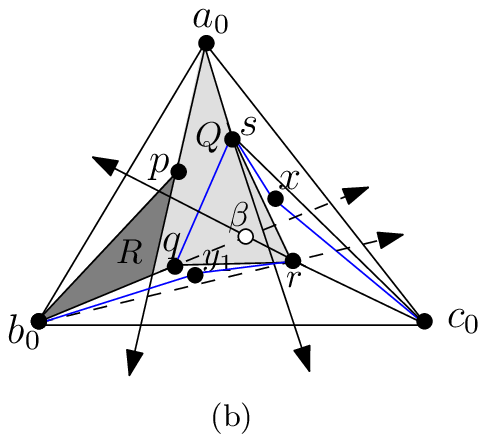}\\
\end{minipage}
\caption{Existence of an empty 7-pseudo-triangle:
(a) $q, r \notin \bbI(Cone(pa_0s))\cap S$, (b) $q\in \bbI(Cone(pa_0s))\cap S$ and $r\notin \bbI(Cone(pa_0s))\cap S$, and (c) $q, r \in \bbI(Cone(pa_0s))\cap S$.}
\label{fig:e(7)empty}
\end{figure*}

Lemma 3 implies that any triangle with more than 4 interior points contains a standard 7-pseudo-triangle.
Now we will obtain an empty 7-pseudo-triangle. Let $S$ be a set of points with $|CH(S)|=3$ and $|\tilde{\bbI}(CH(S))|\geq 5$. Let $\mathcal P_0=a_0pb_0qrc_0s$ be a standard 7-pseudo-triangle contained in $S$ with convex vertices $a_0, b_0, c_0$, and the minimal number of interior points among all the standard 7-pseudo-triangles contained in $S$. Note that the points $a_0, b_0, c_0$ may not be the vertices of $CH(S)$. Now, we have the following three cases:

\begin{description}
\item[{\it Case} 1:] $q, r \notin Cone(pa_0s)\cap S$. Let $\beta$ be the point of intersection of $\overrightarrow{b_0q}$ and $\overrightarrow{c_0r}$, and $x\in
\bbI(\mathcal P_0)\cap S$. If $x\in \bbI(qr\beta)\cap S$,
then $\mathcal P_1=a_0pqxrc_0s$ is a smaller 7-pseudo-triangle contained in
$\mathcal P_0$. Therefore, $\bbI(qr\beta)\cap S$ can be assumed to be empty.
Observe that, if (i) the line segment $xa_0$, and either
of the line segments $xb_0$ or $xc_0$ do not intersect the boundary of $\mathcal P_0$, or (ii) both the line segments $xb_0$ and $xc_0$ intersect the boundary of $\mathcal P_0$, then we can easily construct a
7-pseudo-triangle with fewer interior points than $\mathcal P_0$. Therefore, the shaded
region $Q$ inside $\mathcal P_0$, shown in Figure \ref{fig:e(7)empty}(a), must be
empty. Thus, $x$ lies outside this shaded region and either
$a_0pqrc_0xs$ or $a_0pxb_0qrs$ is a 7-pseudo-triangle with fewer interior
points than $\mathcal P_0$ (Figure \ref{fig:e(7)empty}(a)).

\item[{\it Case} 2:]$q\in Cone(pa_0s)\cap S$ and $r\notin \bbI(Cone(pa_0s))\cap S$.
By similar arguments as in {\it Case} 1, the lightly shaded region $Q$ inside $\mathcal P_0$ shown in Figure \ref{fig:e(7)empty}(b) is empty in $S$. Let $R$ be the deeply shaded region inside $\mathcal P_0$ as shown in Figure \ref{fig:e(7)empty}(b). If $x\in R$, then $a_0pxb_0qrs$ is a 7-pseudo-triangle with fewer interior points than $\mathcal P_0$. Therefore, $Q\cup R$ can be assumed to be empty in $S$. Let $x\in \bbI(\mathcal P_0)\backslash(Q\cup R)\cap S$. The following cases may arise:
\begin{description}
\item[{\it Case} 2.1:]$x$ lies below the line $\overrightarrow{b_0r}$. Then both
$xa_0$ and $xb_0$ intersect the boundary of $\mathcal P_0$ and $a_0pb_0qrxs$ is a 7-pseudo-triangle with fewer interior points.
\item[{\it Case} 2.2:]$x$ lies above the line $\overrightarrow{b_0r}$ but below $\overrightarrow{b_0q}$. Then $a_0pb_0qxc_0s$ is a 7-pseudo-triangle with fewer interior points.
\item[{\it Case} 2.3:]All the interior points of $\mathcal P_0$ must be above the line $\overrightarrow{b_0q}$. If $\bbI(b_0qr)\cap S$ is empty, $a_0qb_0rc_0xs$ is a 7-pseudo-triangle with fewer interior points. Otherwise, $\bbI(b_0qr)\cap S$ is non-empty. Let $Z=(\bbI(b_0qr)\cap S)\cup\{b_0, r\}=\{b_0, y_1, y_2, \ldots, y_m, r\}$, where $m \geq 1$.
\begin{description}
\item[{\it Case} 2.3.1:]$|CH(Z)|\geq 4$, that is, $m \geq 2$. Then $b_0qx_0c_0ry_1\ldots y_m$ forms an empty $k$-mountain, with $k\geq 7$, where $x_0$ is the nearest angular neighbor of $\overrightarrow{b_0q}$ in $\mathcal H(b_0q, a_0)\cap \bbI(\mathcal P_0)$. Hence, $\mathcal P_0$ contains an empty 7-pseudo-triangle from Observation \ref{ob:mm}.
\item[{\it Case} 2.3.2:]$|CH(Z)|=3$. In this case $\bbV(CH(S))=\{b_0, y_1, r\}$, and $\mathcal P_1=b_0y_1rc_0xsq$ is a 7-pseudo-triangle having fewer interior points than $\mathcal P_0$.
\end{description}
\end{description}
\item[{\it Case} 3:] $q, r \in Cone(pa_0s)\cap S$. By similar arguments as in {\it Case} 1 and {\it Case} 2, the lightly shaded region $Q$ inside $\mathcal P_0$, shown in Figure \ref{fig:e(7)empty_n}(a), can be assumed to be empty. Let $x\in (\bbI(\mathcal P_0)\backslash Q)\cap S=(R_1\cup R_2)\cap S$ (see Figure \ref{fig:e(7)empty_n}(a)). W.l.o.g. assume that $x\in R_2\cap S$. The following cases may arise:
\begin{description}
\item[{\it Case} 3.1:]$\bbI(qr\beta)\cap S$ is empty. If $\bbI(b_0qr)\cap S$ is empty, $a_0qb_0rc_0xs$ is a 7-pseudo-triangle with fewer interior points. Otherwise, $\bbI(b_0qr)\cap S$ is non-empty. Let $Z=(\bbI(b_0qr)\cap S)\cup\{b_0, r\}=\{b_0, y_1, y_2, \ldots, y_m, r\}$, where $m \geq 1$.
\begin{description}
\item[{\it Case} 3.1.1:]$|CH(Z)|\geq 4$, that is, $m \geq 2$. Then $b_0qx_0cry_1\ldots y_m$ forms an empty $k$-mountain, with $k\geq 7$, where $x_0$ is the nearest angular neighbor of $\overrightarrow{b_0q}$ in $\mathcal H(b_0q, a_0)\cap \bbI(\mathcal P_0)$. Hence, $\mathcal P_0$ contains an empty 7-pseudo-triangle from Observation \ref{ob:mm}.
\item[{\it Case} 3.1.2:]$|CH(Z)|=3$. In this case $\bbV(CH(S))=\{b_0, y_1, r\}$, and $\mathcal P_1=b_0y_1rc_0xsq$ is a 7-pseudo-triangle having fewer interior points than $\mathcal P_0$.
\end{description}
\item[{\it Case} 3.2:]$\bbI(qr\beta)\cap S$ is non-empty. Let $z \in \bbI(qr\beta)\cap S$. If there exists another point $x\in R_1\cup R_2$ (where $R_1$ and $R_2$ are as shown in Figure \ref{fig:e(7)empty_n}(b)), then either $\mathcal P_1=a_0pxb_0qzr$ (if $x\in R_1$) or $\mathcal P_1=a_0qzrc_0xs$ (if $x\in R_2$) is a 7-pseudo-triangle with $|\bbI(\mathcal P_1)\cap S|<|\bbI(\mathcal P_0)\cap S|$, where $z$ is any point in $\bbI(qr\beta)$. Therefore, assume that $R_1\cup R_2$ is empty in $S$. Let $Z'=\bbV(CH((\bbI(qr\beta)\cap S)\cup\{q, r\}))=\{q, u_1, u_2, \ldots, u_w, r\}$, with $w \geq 1$.
\begin{description}
\item[{\it Case} 3.2.1:]$|Z'|\geq 4$. This means $w\geq 2$, and $a_0pb_0qu_1\ldots u_wr$ is an empty $k$-mountain, with $k\geq 7$. This can be shortened to obtain an empty 7-mountain by Observation \ref{ob:mm}.
\item[{\it Case} 3.2.2:]$|Z'|=3$. Then $Z'=\{q, u_1, r\}$. Now, if $|\bbI(qb_0u_1)\cap S|=0$, then $a_0qb_0u_1rc_0s$ is 7-pseudo-triangle contained in $\mathcal P_0$ with less interior points. Therefore, assume that $|\bbI(qb_0u_1)\cap S|\geq 1$. Let $Z_1=\bbV(CH((\bbI(b_0\beta r)\cap S)\cup\{b_0, r\}))=\{b_0, v_1, v_2, \ldots, v_h, r\}$, with $h \geq 1$. As $|\bbI(qb_0u_1)\cap S|\geq 1$, we have $|Z_1|\geq 4$, that is, $h \geq 2$. Consider the following two cases:
\begin{description}
\item[$|Z_1|\geq 5$:] Then $a_0qb_0v_1\ldots v_hr$ is an empty $k$-mountain, with $k\geq 7$, Hence, $\mathcal P_0$ contains an empty 7-pseudo-triangle from Observation \ref{ob:mm}.
\item[$|Z_1|=4$:] This implies that $Z_1=\{b_0, u_1, v_1, r\}$. If $v_1\in \mathcal H(a_0q, b_0)\cap S$, then $a_0qv_1u_1rc_0s$ is an empty 7-pseudo-triangle. Otherwise, $v_1\in \mathcal H(a_0q, c_0)\cap S$, and $\mathcal P_1:=a_0v_1b_0u_1rc_0s$ is a 7-pseudo-triangle, where $\bbI(\mathcal P_1)\backslash\bbI(b_0u_1v_1)$ is empty in $S$. Now, if $\bbI(b_0u_1v_1)\cap S$ is non-empty, then $\mathcal P_1$ can be easily reduced to an empty 7-pseudo-triangle $a_0xyu_1rc_0s$, where $x$ is the nearest angular neighbor of $\overrightarrow{a_0u_1}$ in $\bbI(b_0u_1v_1)$, and $y$ is the nearest angular neighbor of $\overrightarrow{x\alpha}$ in $\bbI(b_0u_1v_1)$ (see Figure \ref{fig:e(7)empty_n}(c)). Note that $y$ may be equal to $b_0$.
\end{description}
\end{description}
\end{description}
\end{description}

\begin{figure*}[h]
\centering
\begin{minipage}[c]{0.33\textwidth}
\centering
\includegraphics[width=2.95in]
    {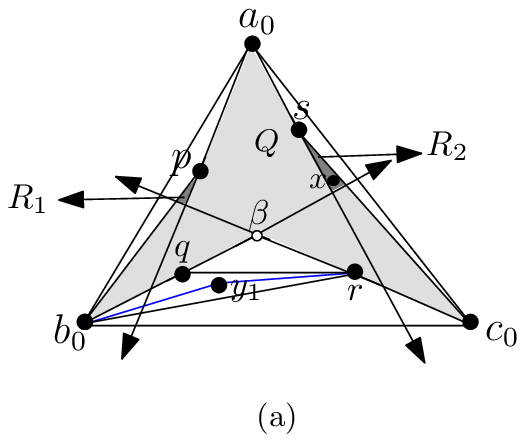}\\
\end{minipage}%
\begin{minipage}[c]{0.33\textwidth}
\centering
\includegraphics[width=2.95in]
    {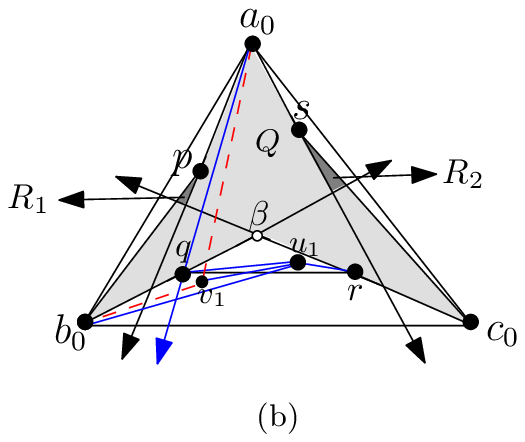}\\
\end{minipage}%
\begin{minipage}[c]{0.33\textwidth}
\centering
\includegraphics[width=1.85in]
    {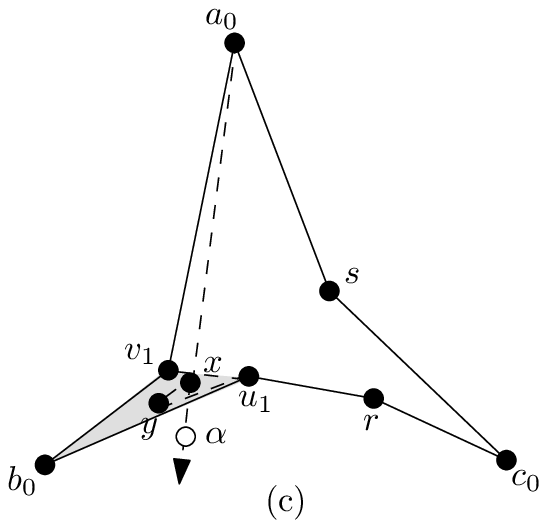}\\
\end{minipage}
\caption{Existence of an empty 7-pseudo-triangle: $q, r \in \bbI(Cone(pa_0s))\cap S$.}
\label{fig:e(7)empty_n}
\end{figure*}

From Lemma \ref{lm:e7_lemma} and the three cases discussed above we obtain that any set of points $S$ in the plane in general position with $|CH(S)|=3$ and $|\tilde{\bbI}(CH(S))|\geq 5$ contains an empty 7-pseudo-triangle.

To show that this is tight, observe that one of the side chains of a 7-pseudo-triangle must have at least three edges. Therefore, any set $S$ with $|CH(S)|=3$ and $|\tilde{\bbI}(CH(S))|=4$  containing a 7-pseudo-triangle must
contain a 4-hole with exactly two consecutive vertices belonging to the vertices of $\bbV(CH(S))$.
It is easy to see that this condition is violated in the point set shown in Figure \ref{fig:example_lower}(a), and the
result follows.

\section{$E(k, \ell)$}
\label{sec:e(k,l)}

As mentioned earlier, $E(k, \ell)$ is the smallest integer such that any set of at least $E(k, \ell)$ points in the plane, no three on a line, contains a $k$-hole or an empty $\ell$-pseudo-triangle. The
existence of $E(k, \ell)$ for all $k, \ell\geq 3$, is a consequence of Theorem \ref{th:valtr} from above \cite{cerny,valtropencupscaps}. However, the general upper bound on $E(k, \ell)$ obtained from
Valtr's \cite{valtropencupscaps} result is double exponential in $k+\ell$. In this section we obtain new bounds on $E(k, \ell)$ for small values of $k$ and $\ell$.

It is clear that $E(k, 3)=E(3, \ell)=3$, for all $k, \ell\geq 3$. Also, $E(k, 4)=k$ for $k\geq 4$ and $E(4, \ell)=5$,  $\ell\geq 5$, since $H(4)=5$. Using the results from the previous sections we will obtain bounds on $E(k, \ell)$ for small values of $k$ and  $\ell$.

We introduce the notion of $\lambda$-convexity, where $\lambda$ is a non-negative integer.
A set $S$ of points in the plane in general position is said to be $\lambda$-{\it convex} if every triangle determined by $S$ contains at most $\lambda$ points of $S$. Valtr \cite{kconvexvaltr,valtropencupscaps} and Kun and Lippner \cite{kunkconvex} proved that for any $\lambda\geq 1$ and $\nu\geq 3$, there is a least integer $N(\lambda, \nu)$ such that any $\lambda$-convex point set of size at least $N(\lambda, \nu)$ contains a $\nu$-hole. The best known upper bound on $N(\lambda, \nu)$ for general $\lambda$ and $\nu$, due to Valtr \cite{valtropencupscaps}, is $N(\lambda, \nu)\leq 2^{{{\lambda+\nu}\choose{\lambda+2}}-1}+1 $, which is double-exponential in $\lambda+\nu$. All known lower bounds on $N(\lambda, \nu)$ are exponential in $\lambda + \nu$.

\subsection{$E(k, 5)$}

In this section we determine the exact value of $E(k, 5)$. We will use Lemma \ref{lm:lm1} and a result of
K\'arolyi et al. \cite{1convexperiodica}.

Although, in general, there is a gap of exponential of $\lambda+\nu$ between the best known upper and lower bounds of $N(\lambda, \nu)$, in the case when $\lambda=1$ much more can be said.
Kun and Lippner \cite{kunkconvex} proved the general upper bound $N(1, \nu)\leq 2^{\lceil(2\nu+5)/3\rceil}$. K\'arolyi et al. \cite{1convexpach} proved that $N(1, \nu)\geq M_\nu$ for odd values of $\nu$, where

$$M_\nu:=\left\{
  \begin{array}{ll}
    2^{(\nu+1)/2}-1, & \hbox{for $\nu\geq 3$ odd;} \\
    \frac{3}{2}2^{\nu/2}-1, & \hbox{for $\nu\geq 4$ even.}
  \end{array}
\right.$$
Finally, K\'arolyi et al. \cite{1convexperiodica} proved that for any $\nu\geq 3$, $N(1, \nu)=M_\nu$.

Using this result, we prove the following theorem:

\begin{theorem}For every positive integer $k\geq 3$, $E(k, 5)=M_k$.
\label{th:E(k,5)}
\end{theorem}

\begin{proof}Let $S$ be a set of $M_k$ points in the plane, in general position.
If there are three points in $S$ such that the triangle determined by them contains more than 1 point of $S$ in its interior, then by Lemma \ref{lm:lm1} $S$ contains an empty 5-pseudo-triangle.
Therefore, $S$ contains a empty 5-pseudo-triangle unless $S$ is 1-convex. However, the maximum size of a 1-convex set not containing a 5-hole is $N(1, k)-1=M_k-1$. Therefore, if $S$ is 1-convex, it always contains a 5-hole. This implies that $E(k, 5)\leq M_k$.

Moreover, if a set is 1-convex, it does not contain any empty 5-pseudo-triangle. This implies that $E(k, 5)> N(1, k)-1=M_k-1$, and it completes the proof that for every $k\geq 3$, $E(k, 5)=M_k$. \hfill $\Box$
\end{proof}

\subsection{$E(5, \ell)$}

It is obvious that $E(5, 3)=3$ and $E(5, 4)=5$. It follows from Theorem \ref{th:E(k,5)} that $E(5, 5)=7$.
In this section we will determine the values of $E(5, \ell)$, for $\ell\geq 6$. We will use the following:

\begin{theorem}[\cite{bbbsdpentagon}] Any set $Z$ of 9 points in the plane in general position, with $|CH(Z)|\geq 4$, contains a 5-hole. \hfill $\Box$
\label{th:5gon9}
\end{theorem}

\begin{figure*}[h]
\centering
\begin{minipage}[c]{0.33\textwidth}
\centering
\includegraphics[width=1.78in]
    {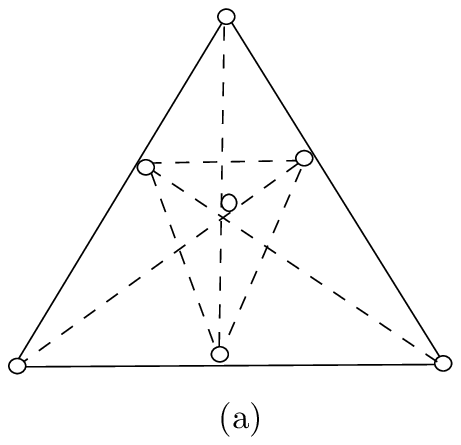}\\
\end{minipage}%
\begin{minipage}[c]{0.3\textwidth}
\centering
\includegraphics[width=1.71in]
    {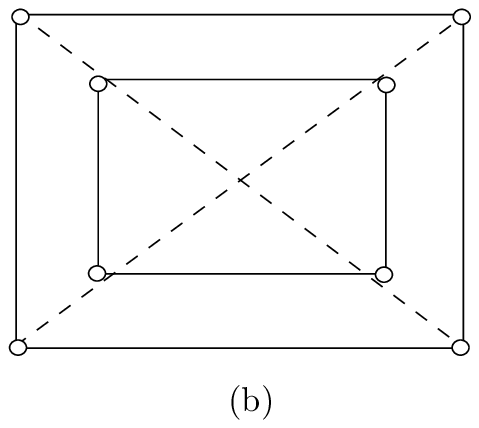}\\
\end{minipage}
\begin{minipage}[c]{0.33\textwidth}
\centering
\includegraphics[width=1.73in]
    {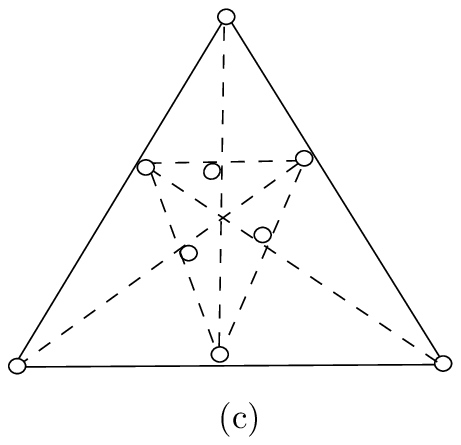}\\
\end{minipage}
\caption{(a) Triangle with 4 interior points and no 7-pseudo-triangle, (b) 8 points with no 5-hole and no 6-pseudo triangle or 7-pseudo-triangle, and (c) 9 points with no 5-hole and no 8-pseudo-triangle.}
\label{fig:example_lower}
\end{figure*}

Using Lemma \ref{lm:lm1} and the above theorem, we will determine the exact values of $E(5, \ell)$ for $\ell\geq 6$.

\begin{theorem}$E(5, 6)=E(5, 7)=9$, and $E(5, \ell)=10$, for $\ell\geq 8$.
\label{th:E(5,l)}
\end{theorem}

\begin{proof}The set of 8 points shown in Figure \ref{fig:example_lower}(b) contains no 5-hole and no
empty 6- or 7-pseudo-triangle.
This implies that $E(5, 6)> 8$ and $E(5, 7)> 8$.

Now, consider a set $S$ of 9 points in general position.
It follows from Theorem \ref{th:5gon9} that $S$ contains a 5-hole whenever
$|CH(S)|\geq 4$. Now, if $|CH(S)|=3$, then $|\tilde{\bbI}(CH(S))|\geq 6$, and the existence of an empty 6-pseudo-triangle and an empty 7-pseudo-triangle in $S$ follows from Lemma \ref{lm2} and Lemma \ref{lm:e7_lemma}, respectively. Therefore, $E(5, 6)\leq 9$ and $E(5, 7)\leq 9$, and together with the lower bound it implies that $E(5, 6)=E(5, 7)=9$.

We know that for $\ell\geq 3$, $E(5, \ell)\leq H(5)=10$, since every set of 10 points in general position contains a 5-hole. The set of 9 points shown in Figure \ref{fig:example_lower}(c) contains no 5-hole and no empty $\ell$-pseudo-triangle for $l\geq 8$. This implies that for $\ell\geq 8$, $E(5, \ell)=10$. \hfill $\Box$
\end{proof}

\subsection{$E(k, 6)$}

In Lemma \ref{lm2} it was proved that any set $S$ of points in the plane in general position with $|CH(S)|=3$ and $|\tilde{\bbI}(CH(S))|\geq 3$ contains an empty standard 6-pseudo-triangle. This implies that $E(k, 6)=N(2, k)\leq 2^{{{k+2}\choose{4}}-1}+1$, since any 2-convex point set cannot contain a 6-pseudo-triangle.

In the special case when $k=6$ we can obtain better bounds. For this reason, we need the following technical lemma:

\begin{figure*}[h]
\centering
\begin{minipage}[c]{0.33\textwidth}
\centering
\includegraphics[width=1.8in]
    {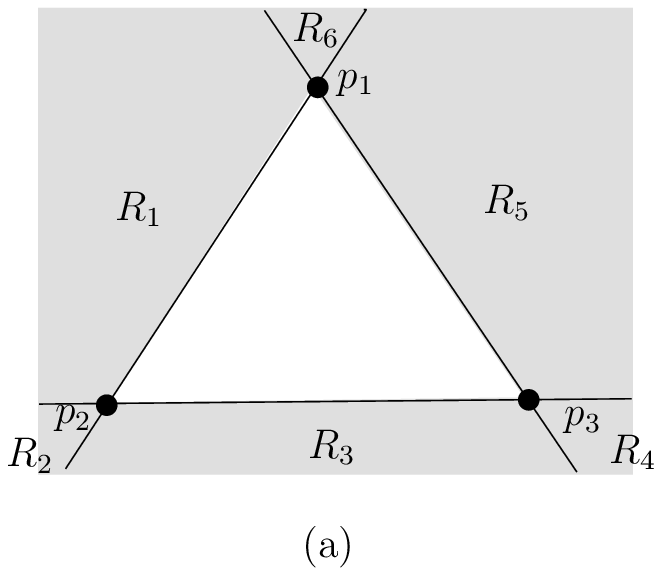}\\
\end{minipage}%
\begin{minipage}[c]{0.3\textwidth}
\centering
\includegraphics[width=1.7in]
    {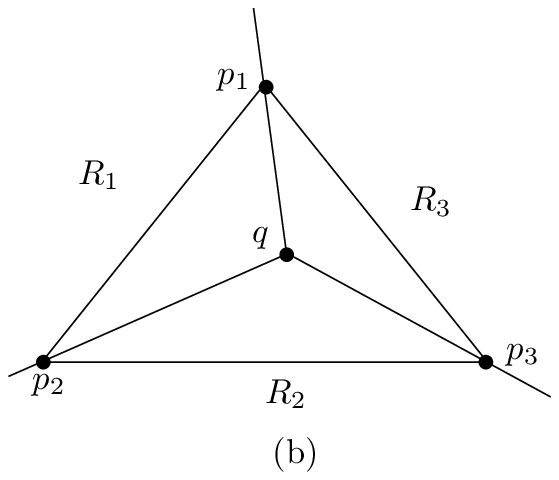}\\
\end{minipage}
\begin{minipage}[c]{0.33\textwidth}
\centering
\includegraphics[width=1.64in]
    {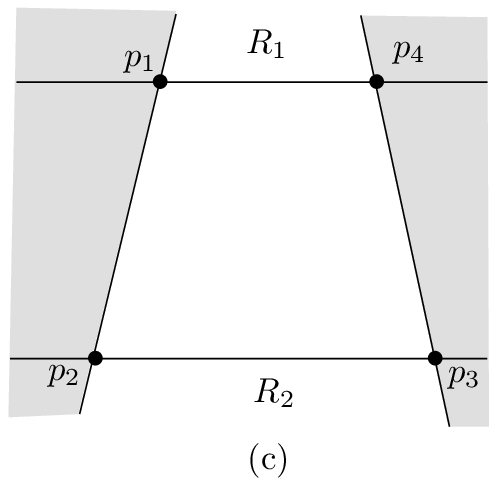}\\
\end{minipage}
\caption{Illustration for the proof of Lemma \ref{lm:lm3}: (a) $|\tilde{\bbI}(CH(Z))|=3$, (b) $|\tilde{\bbI}(CH(Z))|=4$ and $|L\{2, Z\}|=3$, and (c) $|\tilde{\bbI}(CH(Z))|=4$ and $|L\{2, Z\}|=4$.}
\label{fig3II}
\end{figure*}

\begin{lemma}If $Z$ is a set of points in the plane in general position, with $|CH(Z)|\geq 8$ and $|\tilde{\bbI}(CH(Z))|\leq 4$, then $Z$ contains a 6-hole.
\label{lm:lm3}
\end{lemma}

\begin{proof} To prove the lemma it is sufficient to prove the theorem for $|CH(Z)|=8$, since every convex 9-gon can be reduced to a convex 8-gon with at most as many interior points.

If $|\tilde{\bbI}(CH(Z))|=1$, then a 6-hole can be obtained easily. Now, if  $|\tilde{\bbI}(CH(Z))|=2$, then the line joining these two points divides the plane into two half planes one of which must contain at least four points of $\bbV(CH(Z))$. These 4 points together with the two points in $\tilde{\bbI}(CH(Z))$ form a 6-hole.

The remaining two cases are considered separately as follows:

\begin{description}
\item[{\it Case} 1:] $|\tilde{\bbI}(CH(Z))|=3$. Consider the partition of the
exterior of the triangle formed in the second convex layer into disjoint regions $R_i$ as shown in Figure \ref{fig3II} (a). Clearly, $Z$ contains 6-hole, unless the following inequalities hold:

\begin{equation}
| R_{1} | \leq 2, \hspace{1cm} | R_{3} | \leq 2, \hspace{1cm} | R_{5} |\leq 2,
\label{eq:EN2}
\end{equation}
\vspace{-0.25in}
\begin{eqnarray}
| R_{6} | + | R_{1} | + | R_{2} | & \leq & 3, \nonumber \\
| R_{2} | + | R_{3} | + | R_{4} | & \leq & 3, \nonumber \\
| R_{4} | + | R_{5} | + | R_{6} | & \leq & 3.
\label{eq:EN3}
\end{eqnarray}

Summing the inequalities of (\ref{eq:EN3}) and using the fact
$|\bbV(CH(Z))|= 8$ we get {$ |R_{2}| + | R_{4}| + | R_{6}|  \leq 1$}.
Adding this inequality to those from (\ref{eq:EN2}) we get
$\sum_{i=1}^6| R_{i} |\leq 7 < 8=|\bbV(CH(Z))|$, a contradiction.

\item[{\it Case} 2:] $|\tilde{\bbI}(CH(Z))|=4$. We have the following two subcases based on the size of the second layer.
\begin{description}
\item[{\it Case} 2.1:]$|L\{2, Z\}|=3$. Then $|L\{3, Z\}|=1$. Consider the partition of the exterior of $CH(L\{2, Z\})$ into three disjoint regions $R_i$ as shown in Figure \ref{fig3II}(b). Clearly, $S$ contains a 6-hole whenever $|R_i|\geq 3$, for $i\in \{1, 2, 3\}$. Otherwise, $|R_1|+|R_2|+|R_3|\leq 6< 8=|\bbV(CH(Z))|$, a contradiction.
\item[{\it Case} 2.2:]$|L\{2, Z\}|=4$. Let $L\{2, Z\}=\{p_1, p_2, p_3, p_4\}$ be the vertices of the second layer taken in counter-clockwise order. Let $R_1$ and $R_2$ be the shaded regions as shown in Figure \ref{fig3II}(c). It is easy to see that $S$ contains a 6-hole unless $|R_1|+|R_2|\leq 1, |\overline\mathcal H(p_1p_2, p_3)\cap S|\leq 3$, and $|\overline\mathcal H(p_3p_4, p_1)\cap S|\leq 3$. Summing these three inequalities, we get $|\bbV(CH(Z))|\leq 7<8$, a contradiction. \hfill $\Box$
\end{description}
\end{description}
\end{proof}

Using this lemma we prove the following theorem:

\begin{figure*}[h]
\centering
\begin{minipage}[c]{0.5\textwidth}
\centering
\includegraphics[width=2.35in]
    {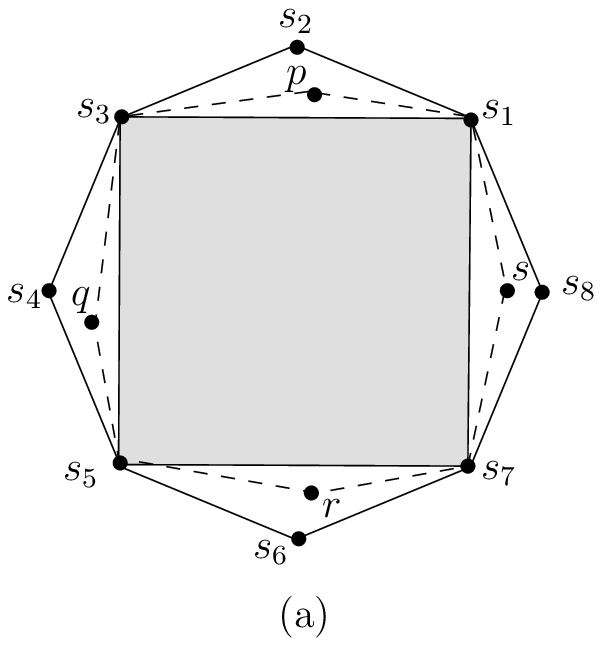}\\
\end{minipage}%
\begin{minipage}[c]{0.5\textwidth}
\centering
\includegraphics[width=2.35in]
    {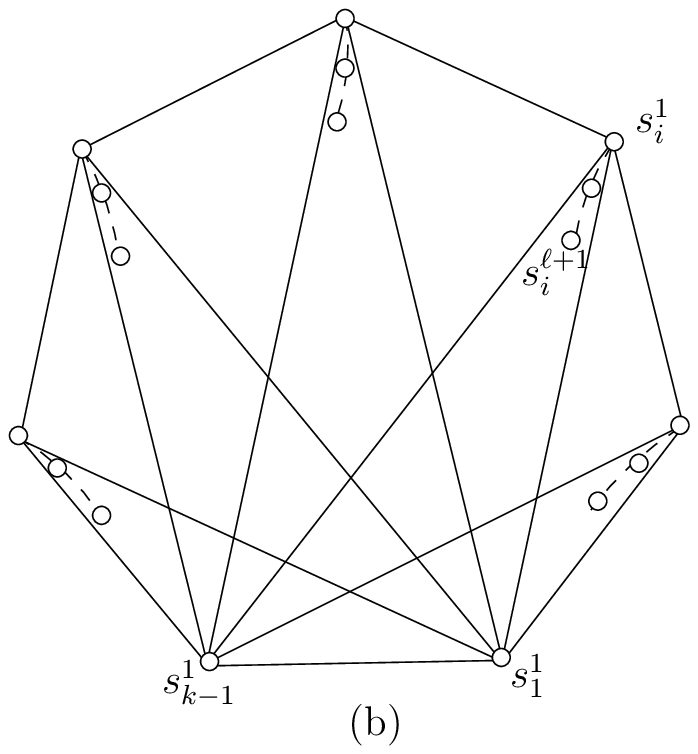}\\
\end{minipage}
\caption{(a) Illustration for the proof of Theorem \ref{th4} and (b) Illustration for the proof of Lemma \ref{lm:bt}.}
\label{fig:e(6,6)}
\end{figure*}

\begin{theorem}
$12 \leq E(6, 6) \leq 18$.
\label{th4}
\end{theorem}
\begin{proof}Using the order-type database, Aichholzer et al. \cite{toth} obtained a set of 11 points that contains neither a convex hexagon nor a 6-pseudo-triangle. This implies that $E(6, 6)\geq 12$.

Consider a set $S$ of 18 points in
general position. Suppose $|CH(S)|=k \leq 7$ and partition $CH(S)$ into $k-2$ triangles
whose vertex set is $\bbV(CH(S))$. Since there are $18-k$ points inside $CH(S)$, there exists a triangle
which has at least $\lceil\frac{18-k}{k-2}\rceil$ points of $S$ inside
it. Observe that $\lceil\frac{18-k}{k-2}\rceil \geq 3$, since $k\leq 7$. Therefore, if $|CH(S)|\leq 7$, then it is possible to find a triangle with at least three interior points and, according to Lemma \ref{lm2}, there exists an empty 6-pseudo-triangle.

Next, suppose that $|CH(S)|=8$. Let $\bbV(CH(S))=\{s_1, s_2, \ldots, s_8\}$, where the vertices are taken in counter-clockwise order. If $|\bbI(s_1s_3s_5s_7)\cap S|\geq 5$, one can find a triangle with at least three interior points and, according to Lemma \ref{lm2}, there is an empty 6-pseudo-triangle. Therefore, suppose that $|\bbI(s_1s_3s_5s_7)\cap S|\leq 4$. Let $p$ be the nearest neighbor of the line segment $s_1s_3$ in $\mathcal H(s_1s_3, s_2)\cap S$. Note that $p$ can be the same as $s_2$, if $\bbI(s_1s_2s_3)\cap S$ is empty. Similarly, let $q, r, s\in S$ be the nearest neighbors of the line segments $s_3s_5$, $s_5s_7$, and $s_7s_1$, respectively (see Figure \ref{fig:e(6,6)}). Observe that the convex octagon $s_1ps_3qs_5rs_7s$ can have at most four points of $S$ inside it. By Lemma \ref{lm:lm3} this convex octagon always contains a 6-hole.

Finally, if $|CH(S)|\geq 9$, then $CH(S)$ can be reduced to a convex octagon with at most as many interior points, and we can apply the same argument as before. Therefore, we have $E(6, 6)\leq 18$.
\hfill $\Box$
\end{proof}

\noindent{\it Remark} 1: Using the order type data-base Aichholzer et al. \cite{toth} observed that there exist precisely 9 out of over 2.33 billion realizable order types of 11 points which contain neither a convex hexagon nor a pseudo-triangle with 6 vertices. Experimenting with Overmars' {\it empty 6-gon program} \cite{overmarsdcg} we were unable to find a set of 12 points which contains no 6-hole and empty 6-pseudo-triangle. In fact, it follows from Lemma \ref{lm2} and the proof of Theorem \ref{th4} that a set $S$ of 12 points contains an empty 6-pseudo-triangle or a 6-hole whenever $|CH(S)|\leq 5$ or $|CH(S)|\geq 8$. Therefore, a set of 12 points without a 6-hole or an empty 6-pseudo-triangle must have $|CH(S)|=6$ or $|CH(S)|=7$.  Although we were unable to geometrically show the existence of a 6-hole or an empty 6-pseudo-triangle in these two cases, experimental evidence motivates us to conjecture that $E(6, 6)=12$. We believe that a very detailed analysis for the different cases that arise when $|CH(S)|$ is either 6 or 7, or some computer-aided enumeration method might be useful in settling the conjecture.

\subsection{Other Improvements and Remarks}
\label{sec:e(k,l)_other}

We now turn our attention to $E(6, \ell)$. Clearly, $E(6, \ell)\leq H(6)$ and $E(6, \ell)\geq N(\ell-4, 6)$, since an $(\ell-4)$-convex set cannot contain an $\ell$-pseudo-triangle. However, when $\ell=7$ we can obtain a better upper bound $E(6, 7)\leq 33$ using Theorem \ref{lm:e7} and a result of Gerken \cite{gerken} which says that any set which contains a 9-gon contains a 6-hole. Consider a set $S$ of 33 points in the plane in general position. Then if $|CH(S)|\geq 9$, $S$ contains a 6-hole and so we can assume that $|CH(S)|=k\leq 8$. $CH(S)$ can be partitioned into $k-2$ triangles whose vertex set is exactly $\bbV(CH(S))$. Since $|\tilde{\bbI}(CH(S)|=33-k$, one of these $k-2$ triangles contains at least $\lceil\frac{33-k}{k-2}\rceil$ interior points. As $k\leq 8$, we have $\lceil\frac{33-k}{k-2}\rceil \geq 5$, and the existence of an empty 7-pseudo-triangle in $S$ follows from Theorem \ref{lm:e7}.\\

\noindent{\it Remark} 2: Note that Theorem \ref{lm:e7} gives a proof of the existence of $E(7, 7)$, which does not use
Theorem \ref{th:valtr}. Valtr's result \cite{kconvexvaltr,valtropencupscaps} implies that any
4-convex set without a 7-hole has at most $N(4,7)-1$ points. So using Theorem \ref{lm:e7} we obtain $E(7, 7)\leq N(4,7)$. Moreover, a three convex set cannot contain a 7-pseudo-triangle, which implies that $E(7, 7)\geq N(3, 7)$.

If one can show that for every integer $\ell\geq 3$ there exists a smallest integer $\Delta(\ell)$ such that any triangle with more than  $\Delta(\ell)$ interior points contains an empty $\ell$-pseudo-triangle, then from Valtr's $\Delta(\ell)$-convexity result it will follow that $E(k, \ell)\leq N(\Delta(\ell), k)$.\\

The bounds obtained on the values $E(k, 5), E(5, \ell), E(k, 6)$, and $E(6, \ell)$ for different values of $k$ and $\ell$ are summarized in Table \ref{table:e(k,l)}.

\begin{table}[h]
\caption{Bounds on $E(k, \ell)$}
\centering
\begin{tabular}{c}
\hline
$E(k, 5)=M_k:=\left\{
  \begin{array}{ll}
    2^{(k+1)/2}-1, & \hbox{for $k\geq 3$ odd;} \\
    \frac{3}{2}2^{k/2}-1, & \hbox{for $k\geq 4$ even.}
  \end{array}
\right.$\\
\hline
\hline
$E(5, \ell)=\left\{
  \begin{array}{ll}
                 3 &\hbox{for $\ell=3$,}\\
                 4 &\hbox{for $\ell=4$,}\\
                 7 &\hbox{for $\ell=5$,}\\
                 9 &\hbox{for $\ell=6$,}\\
                 9 &\hbox{for $\ell=7$,}\\
                 10 &\hbox{for $\ell\geq 8$.}
               \end{array}
\right.$\\
\hline
\hline
$E(k, 6)=N(2, k)= \left\{
  \begin{array}{ll}
                 3 &\hbox{for $k=3$,}\\
                 5 &\hbox{for $k=4$,}\\
                 9 &\hbox{for $k=5$,}\\
                 \left[12, 18\right] &\hbox{for $k=6$,}
               \end{array}
\right.$\\
\hline
\hline
$E(6, \ell)=\left\{
  \begin{array}{ll}
                 3 &\hbox{for $\ell=3$,}\\
                 4 &\hbox{for $\ell=4$,}\\
                 7 &\hbox{for $\ell=5$,}\\
                 \left[12, 18\right] &\hbox{for $\ell=6$,}\\
                 \left[N(3, 6), 33\right] &\hbox{for $\ell=7$,}\\
                 \left[N(\ell-4, 6), H(6)\right] &\hbox{for $\ell\geq 8$.}
               \end{array}
\right.$\\
\hline
\end{tabular}
\label{table:e(k,l)}
\end{table}


\section{$F(k, \ell)$}
\label{sec:f(k,l)}

In the previous sections we have discussed the existence of
{\it empty} convex polygons or pseudo-triangles in point sets. If
the empty condition is dropped, we get another related quantity
$F(k, \ell)$, which we define as the smallest integer such that any set of at least
$F(k, \ell)$ points in the plane, in general position, contains a convex
$k$-gon or a $\ell$-pseudo-triangle. From the Erd\H os-Szekeres theorem
it follows that $F(k, \ell)\leq ES(k)$ for all $k, \ell \geq 3$. Obtaining bounds on $F(k, \ell)$ is also an interesting problem. Aichholzer et al. \cite{toth} showed that $F(6, 6)=12$. Moreover, Aichholzer et al. \cite{toth}
claim that $21 \leq F(7, 7)\leq 23$. In this section, we extend a result of Bisztriczky and Fejes T\'oth \cite{fejestoth}, and obtain the exact values of $F(k, 5)$ and $F(k, 6)$, and non-trivial bounds on $F(k, 7)$.

Bisztriczky and Fejes T\'oth \cite{fejestoth} proved that any $\ell$-convex point set with at least $(k - 3)(\ell +1)+3$ points, not necessarily in general position, contains a convex $k$-gon and the bound is tight. This means that there exists a set of $(k - 3)(\ell +1)+2$ points, not necessarily in general position, which is $\ell$-convex but has no convex $k$-gon.

In the following lemma we generalize the construction of Bisztriczky and Fejes T\'oth \cite{fejestoth} to obtain a set of $(k - 3)(\ell +1)+2$ points {\it in general position} which is $\ell$-convex but has no convex $k$-gon, if $k<\ell/2$.

\begin{lemma}Let $k, \ell$ denote natural numbers such that $k\geq 3$ and $\ell< k/2$. Any $\ell$-convex set of at
least $(k - 3)(\ell +1)+3$ points in the plane in general position contains $k$ points in convex position, and this bound is tight.
\label{lm:bt}
\end{lemma}

\begin{proof}Consider an $\ell$-convex set $S$ of $(k - 3)(\ell +1)+3$ points in the plane in general position. Assume that $|CH(S)|=b\leq k-1$. Consider a triangulation of $CH(S)$ into $b-2$ triangles. Since $S$ is $\ell$-convex, this implies that $|S|\leq b+(b-2)\ell\leq (k-3)(\ell+1)+2$, which is a contradiction.

Now we construct an $\ell$-convex set $Z$ of $(k-3)(\ell+1)+2$ points in general position, which contains no convex $k$-gon. Refer to Figure \ref{fig:e(6,6)}(b). Let $s^1_{1}, s^1_{2}, \ldots, s^1_{k-1}$ be a set of $k-1$ points forming a convex $k-1$-gon ordered in counter-clockwise direction. Consider $Z=\{s^{j}_{i}| i=2, 3, \ldots, k-2; j=1, 2, \ldots, \ell+1\}$, where for every fixed $i \in \{2, 3, \ldots, k-2\}$ and $j \in \{2, 3, \ldots, \ell+1\}$ the point $s^{j}_{i}$ is inside the triangles $s^1_{i-1}s^1_{i}s^1_{i+1}$ and $s_1^1s^1_is^1_{k-1}$. Moreover, depending on whether $k-1$ is even or odd the points in $Z$ satisfy the following properties.
\begin{description}
\item[{\it Case} A:]$k-1=2m$ is even. The set of points $\{s^{j}_{i}| j=2, 3, \ldots, \ell+1\}$ lies on a concave chain $C(s^1_i, s^1_1)$ from $s^1_{i}$ to $s^1_1$, for $i=2, 3, \ldots, m$. Similarly, the set of points $\{s^{j}_{i}| j=2, 3, \ldots, \ell+1\}$ lies on a concave chain $C(s^1_i, s^1_{k-1})$ from $s^1_{i}$ to $s^1_{k-1}$, for $i=m+1, m+2, \ldots, 2m-1~(=k-2)$.

\item[{\it Case} B:]$k-1=2m+1$ is odd. The set of points $\{s^{j}_{i}| j=2, 3, \ldots, \ell+1\}$ lies on a concave chain $C(s^1_i, s^1_1)$ from $s^1_{i}$ to $s^1_1$, for $i=2, 3, \ldots, m$. Similarly, the set of points $\{s^{j}_{i}| j=2, 3, \ldots, \ell+1\}$ lies on a concave chain $C(s^1_i, s^1_{k-1})$ from $s^1_{i}$ to $s^1_{k-1}$, for $i=m+1, m+2, \ldots, 2m~(=k-2)$.
\end{description}

Clearly, $|Z|=(k - 3)(\ell +1)+2$. We shall now show that the set $Z$ constructed above is $\ell$-convex. Consider three distinct points $s^{p}_{i}$, $s^{q}_{j}$, and $s^{r}_{k}$ in $S$. Let $p< q < r$. We will consider three cases:

\begin{description}
\item[{\it Case} 1:]$i=j=k$. Then $\bbI(s^{p}_{i}s^{q}_{j}s^{r}_{k})$ is empty in $Z$.
\item[{\it Case} 2:]$i=j\ne k$. Then the points of $Z$ contained in $\bbI(s^{p}_{i}s^{q}_{j}s^{r}_{k})$
are $s^{p+1}_i, s^{p+2}_i, \ldots, s^{q-1}_i$. Therefore, $|\bbI(s^{p}_{i}s^{q}_{j}s^{r}_{k})\cap S|=q-p-1\leq \ell-1$.
\item[{\it Case} 3:]$i\ne j\ne k$. This implies, the points of $S$ contained in $\bbI(s^{p}_{i}s^{q}_{j}s^{r}_{k})$ are $s^{q+1}_j, s^{q+2}_j, \ldots, s^{\ell+1}_j$. Hence, $|\bbI(s^{p}_{i}s^{q}_{j}s^{r}_{k})\cap S|=\ell-q+1\leq \ell$.
\end{description}

From the above three cases, we conclude that the set $Z$ is $\ell$-convex. It remains to show that it
contains no convex $k$-gon. Let $\mathcal P\subset Z$ be a set of points that form a convex polygon. Let $\mathcal P_i\subset \mathcal P$ be the set of points in $\mathcal P$ which has subscript $i$, for $i\in \{2, 3, \ldots, k-2\}$.

If for all $i\in \{2, 3, \ldots, k-2\}$, $|\mathcal P_i|\leq 1$, then clearly $|\mathcal P|\leq k-1< k$.
Otherwise assume that $|\mathcal P_i|\geq 2$, for at least some $i\in \{2, \ldots, k-2\}$. Note that
because of the orientations of the arrangements of the points in $\mathcal P_i$ along concave chains as described above, there can be at most one subscripts $i$ for which $|\mathcal P_i|\geq 3$. Next, observe that there cannot
be more than 3 subscripts $i$ such that $|\mathcal P_i|\geq 2$, since the set $\mathcal P_i$ is contained in triangles
$s^1_{i-1}s^1_{i}s^1_{i+1}$ and $s_1^1s^1_is^1_{k-1}$. If there are two subscripts $i < j$ such that $|\mathcal P_i|, |\mathcal P_j| \geq 2$, then $\mathcal P \subset \bigcup_{z=i}^j \mathcal P_z$. Therefore, if both $|\mathcal P_i|, |\mathcal P_j|\geq 2$, then none of the points $s^1_{1}$ and $s^1_{k-1}$ can be in $\mathcal P$. Similarly, if $|\mathcal P_i|\geq 2$, then either $\mathcal P \subset \bigcup_{z=i}^{k-1} \mathcal P_z$ or $\mathcal P \subset \bigcup_{z=0}^{i} \mathcal P_z$, and only the point $s^1_{k-1}$ or $s^1_{1}$ can be in $\mathcal P$, respectively.

With these observations, we have the following two cases:

\begin{description}
\item[{\it Case} 1:]Let $|\mathcal P_{i_0}|\geq 3$, for some $i_0$. We  now have the following two cases:
\begin{description}
\item[{\it Case} 1.1:]For all $i\ne i_0$, $|\mathcal P_i|\leq 1$. In this case the largest size of a convex polygon in $Z$ can be obtained by taking all the points in $\mathcal P_{i_0}$, where $i_0=(k-1)/2$ or $i_0=k/2$, depending on whether $k-1$ is even or odd, and one point from each $\mathcal P_i$ on one side of $P_{i_0}$, depending on the curvature of the concave chain at $\mathcal P_{i_0}$. Therefore, the largest possible size of a convex polygon possible is $|\mathcal P|\leq (k-1)/2+\ell$ for $k-1$ even, and $|\mathcal P|\leq k/2+\ell$ for $k-1$ odd. Since $\ell<k/2$, it follows by assumption that $|\mathcal P|< k$.

\item[{\it Case} 1.2:]There exists some $j_0\ne i_0$ such that $|\mathcal P_{j_0}|=2$. In this case the largest size of the convex polygon can be obtained by taking $i_0$ as in {\it Case} 1.1, $j_0=2$ or $j_0=k-2$, and one point each from every $\mathcal P_i$ between $\mathcal P_{i_0}$ and $\mathcal P_{j_0}$. As none of the points $s^1_1$ or $s^1_{k-1}$ can be in $\mathcal P$, it follows that $|\mathcal P|\leq (k-1)/2+\ell$ for $k-1$ even and $|\mathcal P|\leq k/2+\ell$ for $k-1$ odd.
\end{description}

\item[{\it Case} 2:] Let $|\mathcal P_{i_0}|=2$, for some $i_0$, and $\mathcal |P_{j_0}|\leq 2$. If there exits some other $j_0\ne i_0$ such that $|\mathcal P_{j_0}|=2$, then size of a convex polygon that can be found in $Z$ is obtained by taking $i_0=2$ and $j_0=k-2$ (or vice versa) and one point from each $\mathcal P_i$ between $\mathcal P_{i_0}$ and $\mathcal P_{j_0}$. Clearly, the size of the largest convex polygon that can be obtained in this way is $|\mathcal P|\leq k-1$. Otherwise, for all $i\ne i_0$, $|P_{i_0}|=1$, and it is easy to see that $|\mathcal P|\leq k-1$. \hfill $\Box$
\end{description}
\end{proof}



Using this lemma, we now obtain the exact values of $F(k, 5)$ and $F(k, 6)$ in the following theorem:

\begin{theorem}For any positive integer $k\geq 3$, we have
\begin{description}
\item[{\it (i)}]$F(k, 5)=2k-3$ for $k\geq 3$.
\item[{\it (ii)}]$F(k, 6)=3k-6$ for $k\geq 3$.
\end{description}
\label{th:f(k,5)f(k,6)}
\end{theorem}

\begin{proof}Lemma \ref{lm:lm1} implies that any set which has a triangle with 2 interior points has a 5-pseudo-triangle. Moreover, a 1-convex set cannot contain a 5-pseudo-triangle. Therefore, part (i) follows from
Lemma \ref{lm:bt} by with $\ell=1$.

Similarly, Lemma \ref{lm2} implies that any set which has a triangle with 3 interior points has a 6-pseudo-triangle. Moreover, a 2-convex set cannot contain a 5-pseudo-triangle. Therefore, part (ii) follows from
Lemma \ref{lm:bt} by with $\ell=2$.\hfill $\Box$
\end{proof}

In the following theorem, using Lemma \ref{lm:bt} and the results on 7-pseudo-triangles, we obtain new bounds on $F(k, 7)$.

\begin{theorem}$$F(k, 7)=\left\{
  \begin{array}{ll}
                 3 &\hbox{for $k=3$,}\\
                 5 &\hbox{for $k=4$,}\\
                 9 &\hbox{for $k=5$,}\\
                 \left[16, 17\right] &\hbox{for $k=6$,}\\
                 \left[21, 23\right] &\hbox{for $k=7$,}\\
                \left[4k-9, 5k-12\right] &\hbox{for $k\geq 8$.}
               \end{array}
\right.$$
\label{th:f(k,7)}
\end{theorem}

\begin{proof}Using the fact that $ES(4)=5$ and $ES(5)=9$, it is easy to obtain \
$F(4, 7)=5$ and $F(5, 7)=9$, respectively. For $k=6$ we slightly modify the construction in Lemma \ref{lm:bt} to obtain a set of 15 points, shown in Figure \ref{f(6,l)examples}(a) which contains no 6-gon or 7-pseudo-triangle. This example and the fact that $ES(6)=17$ \cite{szekeres}, implies $16 \leq F(6, 7)\leq 17$.

Theorem \ref{lm:e7} implies that any triangle with 5 or more points in its interior contains a 7-pseudo-triangle.
Lemma \ref{lm:bt} with $\ell =4$ implies that any 4-convex set of $5k-12$ points contains a $k$-hole, thus proving that $F(k, 7)\leq 5k-12$. Moreover, any 3-convex point set cannot contain a 7-pseudo-triangle. The lower bound on $F(k, 7)$ now follows from the tightness part of Lemma \ref{lm:bt}, with $\ell=3$ and $k\geq 7$. Therefore, for $k\geq 7$ we have $4k-9\leq F(k, 7) \leq 5k-12$.

For $k=7$, the above inequalities give $19\leq F(7, 7)\leq 23$. As mentioned earlier, the improved lower bound of 21 on $F(7, 7)$ follows from a claim of Aichholzer et al. \cite{aichholzer_survey}. \hfill $\Box$
\end{proof}

\begin{figure*}[h]
\centering
\begin{minipage}[c]{0.5\textwidth}
\centering
\includegraphics[width=2.3in]
    {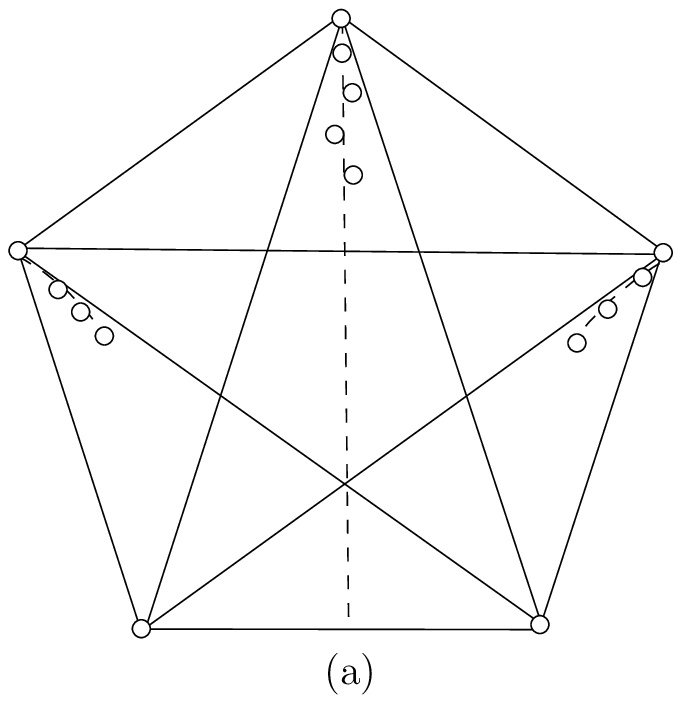}\\
\end{minipage}%
\begin{minipage}[c]{0.5\textwidth}
\centering
\includegraphics[width=2.3in]
   {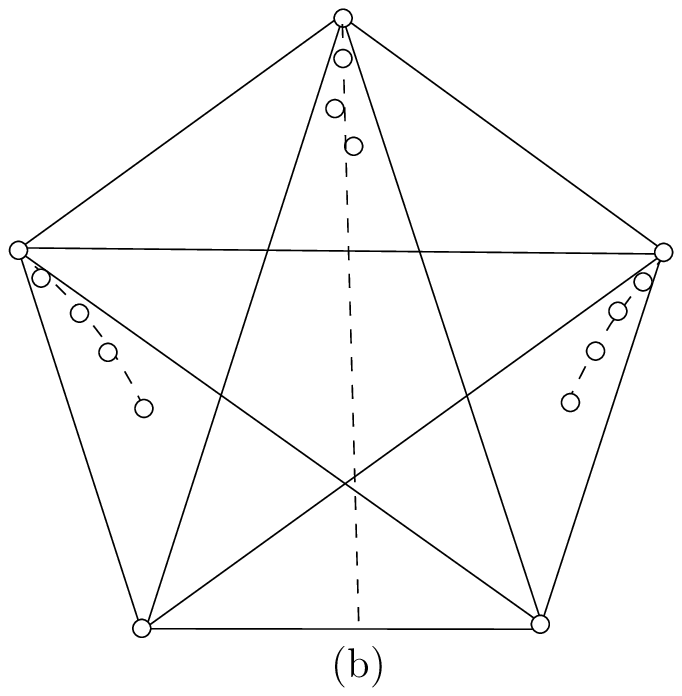}\\
\end{minipage}
\caption{(a) A set of 15 points not containing a 6-gon or a 7-pseudo-triangle, (b) A set of 16 points not containing a 6-gon or an $\ell$-pseudo-triangle for $\ell\geq 8$.}
\label{f(6,l)examples}
\end{figure*}

\noindent{\it Remark} 3: The set of 16 points shown in Figure \ref{f(6,l)examples}(b) is clearly 4-convex. This implies that it cannot contain any $\ell$-pseudo-triangle, for $\ell \geq 8$. Moreover, arguing as in Lemma \ref{lm:bt}, it is easy to see that it contains no convex 6-gon. Since $ES(6)=17$, we have $F(6, \ell)=17$, for $\ell \geq 8$.\\

\noindent{\it Remark} 4: Since an $\ell$-convex point set cannot not contain any $(\ell+4)$ pseudo-triangle, it follows from Lemma \ref{lm:bt} that $F(k, \ell+4)\geq (k-3)(\ell+1)+3$, whenever $\ell< k/2$.\\

The bounds obtained on the values $F(k, 5), F(5, \ell), F(k, 6), F(6, \ell)$, and $F(k, 7)$ for different values of $k$ and $\ell$ are summarized in Table \ref{table:f(k,l)}.

\begin{table}[h]
\caption{Summary of the results}
\centering
\begin{tabular}{c}
\hline
$F(k, 5)=2k-3$\\
\hline
\hline
$F(5, \ell)=\left\{
  \begin{array}{ll}
                 3 &\hbox{for $\ell=3$,}\\
                 4 &\hbox{for $\ell=4$,}\\
                 7 &\hbox{for $\ell=5$,}\\
                 9 &\hbox{for $\ell\geq 6$.}
               \end{array}
\right.$\\
\hline
\hline
$F(k, 6)=3k-6$\\
\hline
\hline
$F(6, \ell)=\left\{
  \begin{array}{ll}
                 3 &\hbox{for $\ell=3$,}\\
                 4 &\hbox{for $\ell=4$,}\\
                 7 &\hbox{for $\ell=5$,}\\
                 12 &\hbox{for $\ell=6$,}\\
                 \left[16, 17\right] &\hbox{for $\ell=7$,}\\
                 17 &\hbox{for $\ell\geq 8$.}
               \end{array}
\right.$\\
\hline
\hline
$F(k, 7)=\left\{
  \begin{array}{ll}
                 3 &\hbox{for $k=3$,}\\
                 5 &\hbox{for $k=4$,}\\
                 9 &\hbox{for $k=5$,}\\
                 \left[16, 17\right] &\hbox{for $k=6$,}\\
                 \left[21, 23\right] &\hbox{for $k=7$,}\\
                \left[4k-9, 5k-12\right] &\hbox{for $k\geq 8$.}
               \end{array}
\right.$\\
\hline
\end{tabular}
\label{table:f(k,l)}
\end{table}

\section{Conclusions}
\label{conclusion}
\vspace{-0.1in}

In this paper we have introduced the quantity $E(k, \ell)$, which denotes the smallest integer such that any set of at least $E(k, \ell)$ points in the plane, no three on a line, contains either an empty convex polygon with $k$ vertices or an empty pseudo-triangle with $\ell$ vertices. The existence of $E(k, \ell)$ for positive integers $k, \ell\geq 3$, is the consequence of a result proved by Valtr \cite{valtropencupscaps}. However, the general upper bound on $E(k, \ell)$ is double-exponential in $k+\ell$. In this paper we prove a series of results regarding the existence
of empty pseudo-triangles in point sets with triangular convex hulls. Using them we determine the exact values of $E(k, 5)$ and $E(5, \ell)$, and prove improved bounds on $E(k, 6)$ and $E(6, \ell)$, for $k, \ell\geq 3$. In particular, we show that $12\leq E(6, 6)\leq 18$ and conjecture the lower bound is, in fact, an equality. Verifying this conjecture and improving the bounds on $E(6, \ell)$, for $\ell\geq 7$ are interesting problems. Proving the existence of $E(k, \ell)$, for $k, \ell \geq 3$, without using Valtr's result \cite{valtropencupscaps}, to obtain a better upper bound in general is also worth investigating.

We have also introduced the quantity $F(k, \ell)$, which is the smallest integer such that any set of at least $F(k, \ell)$ points in the plane, no three on a line, contains a convex polygon with $k$ vertices or an $\ell$-pseudo-triangle. We extend a result of Bisztriczky and T\'oth \cite{fejestoth} and obtain the exact values of $F(k, 5)$ and $F(k, 6)$, and prove bounds on $F(k, 7)$. Obtaining exact values of $F(k, 7)$ for $k\geq 6$, and better general upper bounds on $F(k, \ell)$ for $k, \ell \geq 3$ remains open.\\

\small{\noindent{{\it Acknowledgement:}} The authors wish to thank Bettina Speckmann for her insightful comments on the various properties of pseudo-triangles which led to simplified proofs of some of the results. The authors are indebted to an anonymous referee for valuable comments which has improved both the quality and the presentation of the paper. The authors also thank Tibor Bisztriczky for providing them copies of some of his papers.}

\bibliographystyle{IEEEbib}
\bibliography{strings,refs,manuals}

\end{document}